\documentclass[12pt]{article}    

\usepackage[margin=0.75in]{geometry} 
\usepackage{amsmath,amssymb,amsthm}

\newtheorem{theorem}{Theorem}[section]

\newtheorem{lemma}[theorem]{Lemma}

\newtheorem{remark}{Remark}

\newcommand{\al}{\alpha}
\newcommand{\bt}{\beta}

\newcommand{\s}{\sigma}

\newcommand{\be}{\begin{equation}}
\newcommand{\ee}{\end{equation}}
\newcommand{\bea}{\begin{eqnarray}}
\newcommand{\eea}{\end{eqnarray}}

\numberwithin{equation}{section}
\begin{document}

\title{Differential, Difference and Asymptotic Relations for Pollaczek-Jacobi Type Orthogonal Polynomials and Their Hankel Determinants\renewcommand{\thefootnote}{}\footnote{Submitted on August 6, 2020; Accepted on April 5, 2021; To appear in \textit{Studies in Applied Mathematics}, 2021, https://doi.org/10.1111/sapm.12392}}
\author{Chao Min and Yang Chen}


\date{\today}
\maketitle
\begin{abstract}
In this paper, we study the orthogonal polynomials with respect to a singularly perturbed Pollaczek-Jacobi type weight
$$
w(x,t):=(1-x^2)^\alpha\mathrm{e}^{-\frac{t}{1-x^{2}}},\qquad x\in[-1,1],\;\;\alpha>0,\;\;t>0.
$$
By using the ladder operator approach, we establish the second-order difference equations satisfied by the recurrence coefficient $\beta_n(t)$ and the sub-leading coefficient $\mathrm{p}(n,t)$ of the monic orthogonal polynomials, respectively. We show that the logarithmic derivative of $\beta_n(t)$ can be expressed in terms of a particular Painlev\'{e} V transcendent. The large $n$ asymptotic expansions of $\beta_n(t)$ and $\mathrm{p}(n,t)$ are obtained by using Dyson's Coulomb fluid method together with the related difference equations.

Furthermore, we study the associated Hankel determinant $D_n(t)$ and show that a quantity $\sigma_n(t)$, allied to the logarithmic derivative of $D_n(t)$, can be expressed in terms of the $\sigma$-function of a particular Painlev\'{e} V. The second-order differential and difference equations for $\sigma_n(t)$ are also obtained. In the end, we derive the large $n$ asymptotics of $\sigma_n(t)$ and $D_n(t)$ from their relations with $\beta_n(t)$ and $\mathrm{p}(n,t)$.
\end{abstract}

$\mathbf{Keywords}$: Orthogonal polynomials; Hankel determinant; Pollaczek-Jacobi type weight;

Ladder operators; Painlev\'{e} V; Asymptotic expansions.

$\mathbf{Mathematics\:\: Subject\:\: Classification\:\: 2010}$: 42C05, 33E17, 41A60.

\section{Introduction}
Orthogonal polynomials are of great importance in mathematical physics, random matrix theory, approximation theory, etc. For orthogonal polynomials with weight $w(x)$ supported on $[-1,1]$ and satisfying the Szeg\"{o} condition
\be\label{sz}
\int_{-1}^{1}\frac{\ln w(x)}{\sqrt{1-x^2}}dx>-\infty,
\ee
the classical theory of Szeg\"{o} \cite{Szego} gives a comprehensive description of the large $n$ behavior of the recurrence coefficients and the polynomials.

A weight satisfying the condition (\ref{sz}) is often said to be of the Szeg\"{o} class. The Jacobi weight, $w(x)=(1-x)^\al(1+x)^\bt,\; x\in[-1,1],\; \al,\; \bt>-1$, is a typical example of the Szeg\"{o} class. Kuijlaars et al. \cite{Kuij} considered a modified Jacobi weight
$$
w(x)=(1-x)^\al(1+x)^\bt h(x),\qquad x\in[-1,1],\;\;\; \al,\; \bt>-1,
$$
where $h(x)$ is real analytic and strictly positive on $[-1,1]$. They obtained the large $n$ asymptotics of the orthogonal polynomials, the recurrence coefficients and the associated Hankel determinant by using the steepest descent analysis for Riemann-Hilbert problems.

Zeng, Xu and Zhao \cite{Zeng} studied the asymptotic behavior of the leading coefficients and the recurrence coefficients of the orthonormal polynomials and the Hankel determinant associated with the perturbed Jacobi weight
$$
w(x)=(1-x^2)^\bt(t^2-x^2)^\al h(x),\qquad x\in[-1,1],\;\;\;\bt>-1,\; \al+\bt>-1,\; t>1,
$$
in the sense of a double scaling limit as $n\rightarrow\infty$ and $t\rightarrow 1$. Here the function $h(x)$ satisfies the same condition as above.


However, there are some weights that violate the Szeg\"{o} condition. For example, Chen and Dai \cite{ChenDai} considered a Pollaczek-Jacobi type weight
\be\label{w1}
w(x)=x^\al(1-x)^\bt \mathrm{e}^{-\frac{t}{x}},\qquad x\in [0,1],\;\;\;\al,\;\bt>0,\;t\geq 0,
\ee
and showed that the logarithmic derivative of the Hankel determinant satisfies the Jimbo-Miwa-Okamoto $\s$-form of a particular Painlev\'{e} V. Later, Chen et al. \cite{CCF2019} studied the asymptotic behavior of the orthogonal polynomials, the recurrence coefficients and associated Hankel determinant under a suitable double scaling.

Very recently, by using the Riemann-Hilbert approach, Wang and Fan \cite{Wang} studied the large $n$ asymptotics of the monic orthogonal polynomials with respect to another singularly perturbed Pollaczek-Jacobi type weight
$$
w(x)=x^\al(1-x)^\bt \mathrm{e}^{-\frac{t}{x(1-x)}},\qquad x\in [0,1],\;\;\;\al,\;\bt>0,\;t\geq 0.
$$
Compared to the weight (\ref{w1}), this weight has one more singularity at the edge.
But they did not consider the asymptotics of the recurrence coefficients and the Hankel determinant.

Orthogonal polynomials and Hankel determinants with the singularly perturbed weights have attracted a lot of interests over the past few years, due to the applications in random matrix theory; see \cite{Atkin,BMM,Dai2019,MLC,Xu2015,Xu2016} for reference. This is because Hankel determinants are closely related to the partition functions of the unitary ensembles. The asymptotics of the partition functions usually can be expressed in terms of a particular solution of the Painlev\'{e} equations. In addition, the weights with jump discontinuities and Fisher-Hartwig singularities have also been studied in recent years; see, e.g., \cite{Charlier,CD,Min2019,Min2020,Wu}. See also \cite{DIK} on the asymptotics of Toeplitz, Hankel, and Toeplitz+Hankel determinants with Fisher-Hartwig singularities.

In this paper, we consider the following symmetric Pollaczek-Jacobi type weight with two\\ singularities at the edge, namely,
\be\label{wei}
w(x,t):=(1-x^2)^\al\mathrm{e}^{-\frac{t}{1-x^{2}}},\qquad x\in[-1,1],\;\;\al>0,\;\;t> 0.
\ee
It is easy to see that this weight vanishes infinitely fast at $x=\pm 1$.

Our main purpose is to obtain the large $n$ asymptotic expansions of the recurrence coefficients and the sub-leading coefficients of the monic orthogonal polynomials, and the associated Hankel determinant. We also would like to establish the relation of our problem with the Painlev\'{e} equations in the finite $n$ situation.

Let $P_{n}(x,t),\; n=0,1,2,\ldots,$ be the monic polynomials of degree $n$ orthogonal with respect to the weight (\ref{wei}), i.e.,
\be\label{or}
\int_{-1}^{1}P_{m}(x,t)P_{n}(x,t)w(x,t)dx=h_{n}(t)\delta_{mn},\qquad m, n=0,1,2,\ldots.
\ee
Since the weight $w(x,t)$ is even, we have $P_{n}(-x,t)=(-1)^nP_{n}(x,t)$; see \cite[p. 21]{Chihara}. Then $P_{n}(x,t)$ has the following monomial expansion,
\be\label{expan}
P_{n}(x,t)=x^{n}+\mathrm{p}(n,t)x^{n-2}+\cdots,\qquad n=0,1,2,\ldots,
\ee
Here $\mathrm{p}(n,t)$ denotes the coefficient of $x^{n-2}$, and we will see that it plays a significant role in our problem. We set the initial values of $\mathrm{p}(n,t)$ to be $\mathrm{p}(0,t)=0,\; \mathrm{p}(1,t)=0$.

The orthogonal polynomials $P_{n}(x,t),\; n=0,1,2,\ldots,$ satisfy the following three-term recurrence relation \cite[p. 18-21]{Chihara}
\be\label{rr}
xP_{n}(x,t)=P_{n+1}(x,t)+\beta_{n}(t)P_{n-1}(x,t),
\ee
supplemented by the initial conditions
$$
P_{0}(x,t)=1,\;\;\;\;P_{-1}(x,t)=0.
$$
From (\ref{rr}) we know that the monic orthogonal polynomials are completely determined by the recurrence coefficient $\bt_n(t)$.

The combination of (\ref{or}), (\ref{expan}) and (\ref{rr}) show that the recurrence coefficient $\bt_n(t)$ has two alternative representations:
\bea
\beta_{n}(t)&=&\mathrm{p}(n,t)-\mathrm{p}(n+1,t)\label{be1}\\
&=&\frac{h_{n}(t)}{h_{n-1}(t)}.\label{be2}
\eea
A telescopic sum of (\ref{be1}) produces an important identity
\be\label{sum}
\sum_{j=0}^{n-1}\beta_{j}(t)=-\mathrm{p}(n,t).
\ee
In addition, from (\ref{be1}) we have $\bt_0(t)=0$.

We introduce the Hankel determinant generated by the weight (\ref{wei}),
$$
D_{n}(t):=\det\left(\mu_{i+j}(t)\right)_{i,j=0}^{n-1},
$$
where $\mu_{k}(t),\; k=0, 1, 2,\ldots$ are the moments
\bea\label{mom}
\mu_{k}(t):&=&\int_{-1}^{1}x^{k}w(x,t)dx\nonumber\\[5pt]
&=&\left\{
\begin{aligned}
&0,&k=1,3,5,\ldots;\\
&\mathrm{e}^{-t}\:\Gamma\Big(\frac{k+1}{2}\Big)U\Big(\frac{k+1}{2},-\al,t\Big),&k=0,2,4,\ldots.
\end{aligned}
\right.
\eea
Here $U(a,b,z)$ is the Kummer function of the second kind \cite{NIST}, defined by
\be\label{kum}
U(a,b,z)=\frac{\Gamma(1-b)}{\Gamma(a-b+1)}{}_{1}F_{1}(a;b;z)+\frac{\Gamma(b-1)}{\Gamma(a)}z^{1-b}{}_{1}F_{1}(a-b+1;2-b;z)
\ee
and it has an integral representation
$$
U(a,b,z)=\frac{1}{\Gamma(a)}\int_{0}^{\infty}e^{-zs}s^{a-1}(1+s)^{b-a-1}ds,\qquad \Re a>0,\;\Re z>0.
$$

It is well known that $D_n(t)$ can be expressed as the product of $h_j(t)$ (see \cite[(2.1.6)]{Ismail}),
\be\label{hankel}
D_{n}(t)=\prod_{j=0}^{n-1}h_{j}(t),
\ee
where $h_j(t)$ is defined from the orthogonality (\ref{or}). From (\ref{be2}) and (\ref{hankel}) we have the following relation:
\be\label{re}
\bt_n(t)=\frac{D_{n+1}(t)D_{n-1}(t)}{D_n^2(t)}.
\ee

It is worth pointing out that Hankel determinants play an important role in random matrix theory \cite{Mehta}. Our Hankel determinant $D_n(t)$ can be viewed as the partition function of the singularly perturbed Jacobi unitary ensemble \cite[Corollary 2.1.3]{Ismail}, i.e.,
$$
D_n(t)=\frac{1}{n!}\int_{[-1,1]^n}\prod_{1\leq i<j\leq n}(x_i-x_j)^2\prod_{k=1}^n w(x_k,t)dx_k.
$$
Here $x_1, x_2, \ldots, x_n$, are the eigenvalues of $n\times n$ Hermitian matrices from the ensemble, and the joint probability density function reads,
$$
p(x_1, x_2, \ldots, x_n)\prod_{k=1}^n dx_k=\frac{1}{n!\:D_n(t)}\prod_{1\leq i<j\leq n}(x_i-x_j)^2\prod_{k=1}^n w(x_k,t)dx_k.
$$

Furthermore, we will show that $\s_n(t)$, a quantity allied to the logarithmic derivative of $D_n(t)$ and defined in (\ref{def}), can be expressed in terms of the $\s$-function of a Painlev\'{e} V.

To achieve our main target on the large $n$ asymptotics of $\bt_n(t),\;\mathrm{p}(n,t),\;\s_n(t)$ and $D_n(t)$, first, we derive the second-order difference equations satisfied by $\bt_n(t)$ and $\mathrm{p}(n,t)$ respectively by utilizing the ladder operator approach. Then, we make use of the Coulomb fluid method to obtain the form of the large $n$ asymptotic expansion of $\bt_n(t)$ with the known leading term. The combination gives a full asymptotic expansion of $\bt_n(t)$. The asymptotics of $\mathrm{p}(n,t)$ follows from the corresponding difference equation and the important relation (\ref{be1}). We also find the asymptotics of $\s_n(t)$ from the fact that it can be expressed in terms of $\mathrm{p}(n,t)$. Finally, we derive the asymptotics of the Hankel determinant by connecting it with the free energy and taking advantage of formula (\ref{re}).

The rest of the paper is arranged as follows. In Section 2, we apply the ladder operators to our Pollaczek-Jacobi type weight and obtain two auxiliary quantities $R_n(t)$ and $r_n(t)$. From the compatibility conditions ($S_{1}$), ($S_{2}$) and ($S_{2}'$), we obtain some important identities. Then, we derive the second-order difference equations satisfied by $\bt_n(t)$ and $\mathrm{p}(n,t)$, respectively. Finally, we show the second-order differential equation for the monic orthogonal polynomials $P_n(x,t)$. In Section 3, we prove that the auxiliary quantities $R_n(t)$ and $r_n(t)$ satisfy the coupled Riccati equations, from which we obtain the second-order differential equations for $R_n(t)$ and $r_n(t)$, respectively. We find that $R_n(t)$ is intimately related to a particular Painlev\'{e} V transcendent. Furthermore, we derive the second-order differential and difference equations satisfied by $\s_n(t)$. We also show that this quantity can be expressed in terms of the $\s$-function of a particular Painlev\'{e} V. In Section 4, we study the large $n$ asymptotic expansions of the recurrence coefficient $\bt_n(t)$, the sub-leading coefficient $\mathrm{p}(n,t)$, the log-derivative of the Hankel determinant $\s_n(t)$ and the Hankel determinant $D_n(t)$.

\section{Ladder Operators and Second-Order Difference Equations}
The ladder operator approach has been applied to solve many problems on orthogonal polynomials and Hankel determinants; see, e.g., \cite{BCE,Basor2015,ChenIts,Clarkson3,Dai,Filipuk1,Min2018,Min2019}. Following Chen and Its \cite{ChenIts}, we have the lowering and raising operators for our Pollaczek-Jacobi type orthogonal polynomials:
\be\label{lowering}
\left(\frac{d}{dz}+B_{n}(z)\right)P_{n}(z)=\beta_{n}A_{n}(z)P_{n-1}(z),
\ee
\be\label{raising}
\left(\frac{d}{dz}-B_{n}(z)-\mathrm{v}'(z)\right)P_{n-1}(z)=-A_{n-1}(z)P_{n}(z),
\ee
where $\mathrm{v}(z):=-\ln w(z)$ and
\be\label{an}
A_{n}(z):=\frac{1}{h_{n}}\int_{-1}^{1}\frac{\mathrm{v}'(z)-\mathrm{v}'(y)}{z-y}P_{n}^{2}(y)w(y)dy,
\ee
\be\label{bn}
B_{n}(z):=\frac{1}{h_{n-1}}\int_{-1}^{1}\frac{\mathrm{v}'(z)-\mathrm{v}'(y)}{z-y}P_{n}(y)P_{n-1}(y)w(y)dy.
\ee
Note that we often suppress the $t$-dependence of $P_n(x),\;w(x),\;\bt_n$ and $h_n$ for brevity. We believe that this will not lead to any confusion.

The functions $A_n(z)$ and $B_n(z)$ satisfy the following compatibility conditions valid for $z\in \mathbb{C}\cup\{\infty\}$:
\be
B_{n+1}(z)+B_{n}(z)=z A_{n}(z)-\mathrm{v}'(z), \tag{$S_{1}$}
\ee
\be
1+z(B_{n+1}(z)-B_{n}(z))=\beta_{n+1}A_{n+1}(z)-\beta_{n}A_{n-1}(z), \tag{$S_{2}$}
\ee
\be
B_{n}^{2}(z)+\mathrm{v}'(z)B_{n}(z)+\sum_{j=0}^{n-1}A_{j}(z)=\beta_{n}A_{n}(z)A_{n-1}(z). \tag{$S_{2}'$}
\ee

In addition, eliminating $P_{n-1}(z)$ from (\ref{lowering}) and (\ref{raising}) shows that
$P_{n}(z)$ satisfies the second-order linear ordinary differential equation
\be\label{general}
P_{n}''(z)-\left(\mathrm{v}'(z)+\frac{A_{n}'(z)}{A_{n}(z)}\right)P_{n}'(z)+\left(B_{n}'(z)-B_{n}(z)\frac{A_{n}'(z)}{A_{n}(z)}
+\sum_{j=0}^{n-1}A_{j}(z)\right)P_{n}(z)=0,
\ee
where use has been made of ($S_{2}'$).

For the weight function given in (\ref{wei}), we find
\be\label{pt}
\mathrm{v}(z)=-\ln w(z)=\frac{t}{1-z^{2}}-\al\ln(1-z^2).
\ee
Hence,
$$
\mathrm{v}'(z)=\frac{2\al z}{1-z^2}+\frac{2tz}{(1-z^2)^2}
$$
and
\be\label{vp}
\frac{\mathrm{v}'(z)-\mathrm{v}'(y)}{z-y}=\frac{2\al(1+zy)}{(1-z^2)(1-y^2)}+\frac{2t\left[1-(z^2-2)zy-z^2y^2-zy^3\right]}{(1-z^2)^2(1-y^2)^2}.
\ee
Since the right hand side of (\ref{vp}) is rational, $A_n(z)$ and $B_n(z)$ should be also rational from their definitions in (\ref{an}) and (\ref{bn}). More precisely, we have the following lemma. The proof has been omitted and we refer the reader to \cite{MLC} for a similar derivation.
\begin{lemma}
We have
\be\label{anz}
A_{n}(z)=\frac{2n+1+2\al}{1-z^2}+\frac{R_{n}(t)}{(1-z^2)^2},
\ee
\be\label{bnz}
B_{n}(z)=\frac{nz}{1-z^2}+\frac{z\:r_{n}(t)}{(1-z^2)^2},
\ee
where
\be\label{Rnt}
R_{n}(t):=\frac{2t}{h_{n}}\int_{-1}^{1}\frac{1}{1-y^2}P_{n}^{2}(y)w(y) dy,
\ee
\be\label{rnt}
r_{n}(t):=\frac{2t}{h_{n-1}}\int_{-1}^{1}\frac{y}{1-y^2}P_{n}(y)P_{n-1}(y)w(y) dy.
\ee
\end{lemma}

\begin{remark}
For $n=0$, we find from the definitions of $R_n(t)$ and $r_n(t)$ that
\be\label{R0}
R_0(t)=2t\:\frac{U\big(\frac{1}{2},1-\al,t\big)}{U\big(\frac{1}{2},-\al,t\big)},
\ee
$$
r_0(t)=0,
$$
where $U(a,b,z)$ is the Kummer function of the second kind \cite{NIST}.
\end{remark}
Substituting (\ref{anz}) and (\ref{bnz}) into ($S_{1}$), we obtain
\be\label{s1}
r_{n+1}(t)+r_{n}(t)=R_{n}(t)-2t.
\ee
From ($S_{2}$) we have the following two equations:
\be\label{s21}
r_{n+1}(t)-r_n(t)=\bt_{n+1}R_{n+1}(t)-\bt_{n}R_{n-1}(t),
\ee
\be\label{s22}
r_{n+1}(t)-r_n(t)-1=(2n-1+2\al)\bt_n-(2n+3+2\al)\bt_{n+1}.
\ee
We write (\ref{s22}) in another form
$$
r_{n+1}(t)-r_n(t)-1=(2n-1+2\al)\bt_n-(2n+1+2\al)\bt_{n+1}-2\bt_{n+1}.
$$
Replacing $n$ by $j$ in the above and taking a telescopic sum from $j=0$ to $j=n-1$ produces an important identity
\be\label{s2}
r_n(t)=n-(2n+1+2\al)\bt_n+2\mathrm{p}(n,t),
\ee
where we have used (\ref{sum}) and the initial conditions $\bt_0=0, r_0(t)=0$.

Finally, from ($S_{2}'$), we obtain the following three equations:
\be\label{s2p1}
r_{n}^2(t)+2t\:r_n(t)=\beta_{n}R_{n}(t)R_{n-1}(t),
\ee
\be\label{s2p2}
r_{n}^2(t)+2(t-n-\al)r_{n}(t)-2nt+(2n+1+2\al)\beta_{n}R_{n-1}(t)+(2n-1+2\al)\beta_{n}R_{n}(t)=0,
\ee
\be\label{s2p3}
n(n+2\al-2t)-2(n+\al)r_n(t)+\sum_{j=0}^{n-1}R_{j}(t)
=(2n+1+2\al)(2n-1+2\al)\beta_{n}.
\ee
\begin{theorem}
The recurrence coefficient $\bt_n$ satisfies the following second-order nonlinear difference equation:
\begin{small}
\bea\label{btd}
&&\Big\{\left[68 (n+\alpha)^2-9\right]\bt_n^3+\big[12-80(n+\al)^2+ (14 n+5+14 \alpha)\tilde{\bt}_{n-1}+(14 n-5+14 \alpha)\tilde{\bt}_{n+1}\big]\bt_n^2\nonumber\\
&+&\big[24(n+\al)^2+4t(t-2\al)-3-2 (2 n+1+2 \alpha)\tilde{\beta}_{n-1}-2(2 n-1+2 \alpha) \tilde{\bt}_{n+1}+\tilde{\beta}_{n-1} \tilde{\beta}_{n+1}\big]\bt_n\nonumber\\
&-&2\left[(n+\al)^2-t\al\right]\Big\}^2=4\Big\{2 (2 n-1+2 \alpha) (2 n+1+2 \alpha)\bt_n^2+\big[(2 n+1+2 \alpha) \tilde{\beta}_{n-1}\nonumber\\
&+&(2 n-1+2 \alpha) \tilde{\bt}_{n+1}-2 (2 n-1+2 \alpha) (2 n+1+2 \alpha)\big]\bt_n+(n+\al)^2+t(t-2\al)\Big\}\nonumber\\
&\times&\Big\{12 (n+\alpha) \beta_n^2+\big[ \tilde{\beta}_{n-1}+ \tilde{\beta}_{n+1}-8 (n+\alpha)\big]\beta_n+n+\al\Big\}^2,
\eea
\end{small}
where
$$
\tilde{\bt}_{n-1}:=(2n-3+2\al)\bt_{n-1}(t),
$$
$$
\tilde{\bt}_{n+1}:=(2n+3+2\al)\bt_{n+1}(t).
$$
\end{theorem}
\begin{proof}
From (\ref{s2}), we have
\be\label{e1}
r_{n+1}(t)=n+1-(2n+3+2\al)\bt_{n+1}+2\mathrm{p}(n+1,t),
\ee
\be\label{e2}
r_{n-1}(t)=n-1-(2n-1+2\al)\bt_{n-1}+2\mathrm{p}(n-1,t).
\ee
Using (\ref{be1}), it follows that
\be\label{e3}
\mathrm{p}(n+1,t)=\mathrm{p}(n,t)-\bt_n,
\ee
\be\label{e4}
\mathrm{p}(n-1,t)=\mathrm{p}(n,t)+\bt_{n-1}.
\ee
Substituting (\ref{e3}) into (\ref{e1}) and (\ref{e4}) into (\ref{e2}) respectively, we get
\be\label{e5}
r_{n+1}(t)=n+1-(2n+3+2\al)\bt_{n+1}-2\bt_n+2\mathrm{p}(n,t),
\ee
\be\label{e6}
r_{n-1}(t)=n-1-(2n-3+2\al)\bt_{n-1}+2\mathrm{p}(n,t).
\ee
In view of (\ref{s1}), we have
\be\label{e7}
R_n(t)=2t+r_{n+1}(t)+r_n(t),
\ee
\be\label{e8}
R_{n-1}(t)=2t+r_{n}(t)+r_{n-1}(t).
\ee
Inserting (\ref{e5}) into (\ref{e7}) and (\ref{e6}) into (\ref{e8}) respectively, we find
\be\label{e9}
R_n(t)=n+1+2 t+r_n(t)+2 \mathrm{p}(n,t)-(2n+3+2 \alpha) \beta_{n+1}-2 \beta_n,
\ee
\be\label{e10}
R_{n-1}(t)=n-1+2 t+r_n(t)+2 \mathrm{p}(n,t)-(2n-3+2 \alpha) \beta_{n-1}.
\ee
Substituting (\ref{e9}) and (\ref{e10}) into (\ref{s2p1}) and (\ref{s2p2}), we obtain
\bea\label{e11}
r_{n}^2(t)+2t\:r_n(t)&=&\beta_{n}\left[n+1+2 t+r_n(t)+2 \mathrm{p}(n,t)-(2n+3+2 \alpha) \beta_{n+1}-2 \beta_n\right]\nonumber\\
&\times&\left[n-1+2 t+r_n(t)+2 \mathrm{p}(n,t)-(2n-3+2 \alpha) \beta_{n-1}\right],
\eea
\bea\label{e12}
&&r_{n}^2(t)+2(t-n-\al)r_{n}(t)-2nt+(2n+1+2\al)\beta_{n}\big[n-1+2 t+r_n(t)+2 \mathrm{p}(n,t)\nonumber\\
&-&(2n-3+2 \alpha) \beta_{n-1}\big]+(2n-1+2\al)\beta_{n}\big[n+1+2 t+r_n(t)+2 \mathrm{p}(n,t)\nonumber\\
&-&(2n+3+2 \alpha) \beta_{n+1}-2 \beta_n\big]=0.
\eea

Equations (\ref{s2}), (\ref{e11}) and (\ref{e12}) can be regarded as a system of nonlinear equations satisfied by $\bt_n,\; r_n(t)$ and $\mathrm{p}(n,t)$. Now we are ready to derive the second-order difference equation for $\bt_n$ from this system. We begin with expressing $\mathrm{p}(n,t)$ in terms of $\bt_n$ and $r_n(t)$ from (\ref{s2}),
$$
2\mathrm{p}(n,t)=(2 n+1+2 \alpha) \beta_n+r_n(t)-n.
$$
Inserting it into (\ref{e11}) and (\ref{e12}) respectively, we get the following two equations:
\bea\label{e13}
&&r_n^2+2t\: r_n-\beta_n\left[(2n+1+2\al)\bt_n-(2n-3+2 \alpha) \beta_{n-1}+2 r_n+2 t-1\right]\nonumber\\
&\times&\left[(2n-1+2 \alpha) \beta_n-(2n+3+2 \alpha) \beta_{n+1}+2 r_n+2 t+1\right]=0,
\eea
\bea\label{e14}
&&r_n^2+2 r_n \left[t-(n+\alpha) (1-4 \beta_n)\right]+2 \beta_n^2\left[4(n+\al)^2+1\right]-\bt_n\big[2-8t(n+\al)\nonumber\\
&+&(2 n+1+2 \alpha )(2 n-3+2 \alpha) \bt_{n-1}+(2 n-1+2 \alpha ) (2 n+3+2 \alpha )\bt_{n+1}\big]-2nt=0.
\eea
Note that equation (\ref{e14}) may be viewed as a quadratic equation for $r_n(t)$. Substituting either solution into (\ref{e13}), we obtain (\ref{btd}) after clearing the square root.
\end{proof}

\begin{theorem}
The sub-leading coefficient of the monic orthogonal polynomials, $\mathrm{p}(n):=\mathrm{p}(n,t)$, satisfies the following second-order nonlinear difference equation:
\begin{small}
\bea\label{pnd}
&&\left(n+2\mathrm{p}(n)-\tilde{\mathrm{p}}(n+1)\right)^3+\left(n+2\mathrm{p}(n)-\tilde{\mathrm{p}}(n+1)\right)^2
\left(n-2+4t-\tilde{\mathrm{p}}(n-1)+\tilde{\mathrm{p}}(n+1)\right)\nonumber\\
&-&2\left(n+2\mathrm{p}(n)-\tilde{\mathrm{p}}(n+1)\right)\big[n^2-n(1-\alpha)-\alpha+2 t-2 t^2-\tilde{\mathrm{p}}(n-1) \left(n+\al-t-\tilde{\mathrm{p}}(n+1)\right)\nonumber\\
&-&(n-1+2 t)\tilde{\mathrm{p}}(n+1)\big]-\left(n-1+2t-\tilde{\mathrm{p}}(n-1)\right)\big[2nt-\tilde{\mathrm{p}}(n+1)\left(n-1+2t-\tilde{\mathrm{p}}(n-1)\right)\big]\nonumber\\
&+&4 \mathrm{p}^2(n) \tilde{\mathrm{p}}(n+1)+2\mathrm{p}(n)\big[\left(n+2 \mathrm{p}(n)-\tilde{\mathrm{p}}(n+1)\right)^2+2\tilde{\mathrm{p}}(n+1) \left(n-1+2t-\tilde{\mathrm{p}}(n-1)\right)\nonumber\\
&-&2 \left(n+2\mathrm{p}(n)-\tilde{\mathrm{p}}(n+1)\right) \left(n+\alpha-t -\tilde{\mathrm{p}}(n+1)\right)-2 n t\big]=0,
\eea
\end{small}
where
$$
\tilde{\mathrm{p}}(n-1):=(2n-3+2\al)(\mathrm{p}(n-1,t)-\mathrm{p}(n,t)),
$$
$$
\tilde{\mathrm{p}}(n+1):=(2n+1+2\al)(\mathrm{p}(n,t)-\mathrm{p}(n+1,t)).
$$
\end{theorem}

\begin{proof}
Eliminating $\bt_{n+1}$ from equations (\ref{e11}) and (\ref{e12}), and then replacing $r_n(t)$ by (\ref{s2}), we get an equation for $\bt_n,\; \bt_{n-1}$ and $\mathrm{p}(n,t)$. Using the following relations from (\ref{be1}) to eliminate $\bt_n$ and $\bt_{n-1}$,
$$
\bt_{n}=\mathrm{p}(n,t)-\mathrm{p}(n+1,t),
$$
$$
\bt_{n-1}=\mathrm{p}(n-1,t)-\mathrm{p}(n,t),
$$
we finally obtain the desired result. Noting that, without the first step to eliminate $\bt_{n+1}$, we would obtain the third-order difference equation for $\mathrm{p}(n,t)$.
\end{proof}

\begin{remark}
One can also derive the second-order difference equations satisfied by the auxiliary quantities $R_n(t)$ and $r_n(t)$ from the system of algebraic equations (\ref{s1}), (\ref{s2p1}) and (\ref{s2p2}), following the similar procedure in \cite{MLC}.
\end{remark}

In the end of this section, we show the second-order differential equation satisfied by $P_n(z)$, with the coefficients being rational functions with singular points at $z=\pm 1$. Moreover, we would express the coefficients in terms of $\bt_n$ and $\mathrm{p}(n,t)$.
\begin{theorem}
The monic orthogonal polynomials $P_n(z),\; n=0,1,2,\ldots,$ satisfy the following second-order differential equation:
$$
P_n''(z)-\left(\mathrm{v}'(z)+\frac{A_{n}'(z)}{A_{n}(z)}\right)P_n'(z)+\left(B_{n}'(z)-B_{n}(z)\frac{A_{n}'(z)}{A_{n}(z)}
+\sum_{j=0}^{n-1}A_{j}(z)\right)P_n(z)=0,
$$
where
\be\label{anz1}
A_n(z)=\frac{2 n+1+2\al}{1-z^2}+\frac{2 n+1+2 t-(2 n+3+2 \alpha) (\beta_n+\beta_{n+1})+4 \mathrm{p}(n,t)}{(1-z^2)^2},
\ee
\be\label{bnz1}
B_n(z)=\frac{n z}{1-z^2}+\frac{z \left[n-(2 n+1+2 \alpha) \beta_n+2 \mathrm{p}(n,t)\right]}{(1-z^2)^2},
\ee
\be\label{sum1}
\sum_{j=0}^{n-1}A_{j}(z)=\frac{n^2+2 n\alpha}{1-z^2}+\frac{n (n+2 t)-(2 n+1+2 \alpha) \beta_n+4 (n+\al) \mathrm{p}(n,t)}{(1-z^2)^2},
\ee
and
$$
\mathrm{v}'(z)=\frac{2\al z}{1-z^2}+\frac{2tz}{(1-z^2)^2}.
$$
\end{theorem}
\begin{proof}
The general form of the second-order differential equation satisfied by the monic orthogonal polynomials has been given in (\ref{general}). The remaining task is to express the coefficients of $P_n(z)$ and $P_n'(z)$ in terms of $\bt_n$ and $\mathrm{p}(n,t)$.

The combination of (\ref{s1}) and (\ref{s2}) gives the expression of $R_n(t)$ in terms of $\bt_n$ and $\mathrm{p}(n,t)$:
\be\label{exp1}
R_n(t)=2n+1+2t-(2n+3+2\al)(\bt_n+\bt_{n+1})+4\mathrm{p}(n,t),
\ee
where we have used the fact $\mathrm{p}(n+1,t)=\mathrm{p}(n,t)-\bt_n$.
Substituting (\ref{exp1}) into (\ref{anz}) and (\ref{s2}) into (\ref{bnz}) respectively, we obtain (\ref{anz1}) and (\ref{bnz1}).

From (\ref{anz}) we have
$$
\sum_{j=0}^{n-1}A_{j}(z)=\frac{n^2+2n\al}{1-z^2}+\frac{\sum_{j=0}^{n-1}R_j(t)}{(1-z^2)^2}.
$$
Using (\ref{s2p3}) to eliminate $\sum_{j=0}^{n-1}R_j(t)$ and in view of (\ref{s2}), we obtain (\ref{sum1}). This completes the proof. We mention that $\mathrm{p}(n,t)$ can also be expressed in terms of the $\bt_j$ via $\mathrm{p}(n,t)=-\sum_{j=0}^{n-1}\beta_{j}$.
\end{proof}

\section{$t$ Evolution and Painlev\'{e} V}
Recall that our weight function depends on $t$. As a consequence, the recurrence coefficient, the sub-leading coefficient and the auxiliary quantities all depend on $t$.
The objective of this section is to establish the relationships between the auxiliary quantities and the derivatives with respect to $t$ of the key quantities $\bt_n,\;\mathrm{p}(n,t)$ and $\ln h_n$. This, in turn, will allow us to obtain the coupled Riccati equations satisfied by the auxiliary quantities $R_n(t)$ and $r_n(t)$. Based on these results, we find that $R_n(t)$, up to a simple linear transformation, is the solution of a particular Painlev\'{e} V equation.

Following the similar procedure in \cite{MLC}, we start from taking derivatives with respect to $t$ in the equations
$$
\int_{-1}^{1}P_{n}^2(x,t)(1-x^2)^\al\mathrm{e}^{-\frac{t}{1-x^{2}}}dx=h_{n}(t),\qquad n=0,1,2,\ldots
$$
and
$$
\int_{-1}^{1}P_{n}(x,t)P_{n-2}(x,t)(1-x^2)^\al\mathrm{e}^{-\frac{t}{1-x^{2}}}dx=0,\qquad n=1,2,\ldots,
$$
respectively. This leads to two important relations:
\be\label{eq1}
2t \frac{d}{dt}\ln h_{n}(t)=-R_{n}(t),
\ee
\be\label{pnt}
2t\frac{d}{dt}\mathrm{p}(n,t)=r_n(t)-\beta_{n}R_{n}(t).
\ee
Moreover, the combination of (\ref{be2}) and (\ref{eq1}) shows that
$$
2t\frac{d}{dt}\ln\beta_{n}(t)=R_{n-1}(t)-R_{n}(t).
$$
That is,
\be\label{eq2}
2t\beta_{n}'(t)=\beta_{n}R_{n-1}(t)-\beta_{n}R_{n}(t).
\ee

Similarly as in \cite{MLC}, by making use of (\ref{pnt}), (\ref{eq2}) and the results from the compatibility conditions in Section 2, we have the following lemma.
\begin{lemma} The auxiliary quantities $r_n(t)$ and $R_n(t)$ satisfy the coupled Riccati equations:
\be\label{ric1}
2t\: r_n'(t)=2nt-r_n^2(t)+2(n+\al+1-t)r_n(t)-\frac{2(2n+1+2\al)(r_n^2(t)+2t\:r_n(t))}{R_n(t)},
\ee
\be\label{ric2}
2tR_n'(t)=R_n^2(t)+2(n+\al+1-t)R_n(t)-2r_n(t)(2n+1+2\al+R_n(t))-2t(2n+1+2\al).
\ee
\end{lemma}

\begin{theorem}
The auxiliary quantities $R_n(t)$ and $r_n(t)$ satisfy the following second-order nonlinear ordinary differential equations:
\bea\label{Rn}
&&8 t^2 R_n (2 n+1+2 \alpha+R_n)R_n''-4 t^2 (4 n+2+4 \alpha+3 R_n)(R_n')^2+8 t R_n (2n+1+2 \alpha+R_n)R_n'\nonumber\\
&-&R_n^5-2(2 n+1+2 \alpha)R_n^4-4\left[(n+\al)(n+1+\al)-t(t-2\al)\right]R_n^3+16 t (2 n+1+2 \alpha) (t-\alpha )R_n^2\nonumber\\
&+&4 t (2 n+1+2 \alpha)^2 (5 t-2 \alpha )R_n+8 t^2 (2 n+1+2 \alpha)^3=0,
\eea\\[-50pt]
\bea\label{rn}
&&4 t^2 r_n (2t+r_n)r_n''-4 t^2 (t+r_n)(r_n')^2+4 t\: r_n^2\: r_n'-r_n^5-(2 n+2\al+5 t)r_n^4-8 t  (n+\alpha+t)r_n^3\nonumber\\
&-&4 t \left[(t+\al)^2+n (2 t+\al)-1\right]r_n^2+4 n^2 t^2\: r_n+4 n^2 t^3=0.
\eea
Let
$$
W_n(t):=\frac{2n+1+2\al+R_n(t)}{2n+1+2\al}.
$$
Then $W_n(t)$ satisfies the Painlev\'{e} V equation \cite{Gromak}
\be\label{pv}
W_{n}''=\frac{(3 W_n-1) (W_n')^2}{2W_n (W_n-1) }-\frac{W_n'}{t}+\frac{(W_n-1)^2 }{t^2}\left(\mu_1 W_n +\frac{\mu_2}{W_n}\right)+\frac{\mu_3 W_n}{t}+\frac{\mu_4 W_n(W_n+1) }{W_n-1},
\ee
with
\be\label{pa}
\mu_{1}=\frac{(2n+1+2 \alpha)^2}{8},\qquad \mu_2=-\frac{1}{8},\qquad \mu_3=\al,\qquad \mu_4=-\frac{1}{2}.
\ee
\end{theorem}
\begin{proof}
Solving for $r_n(t)$ from (\ref{ric2}) and substituting it into (\ref{ric1}), we obtain (\ref{Rn}). On the other hand, solving for $R_n(t)$ from (\ref{ric1}) and substituting it into (\ref{ric2}), we arrive at (\ref{rn}). With the given linear transformation, equation (\ref{Rn}) turns into (\ref{pv}).
\end{proof}
\begin{remark}
We point out that, the paper \cite{Magnus} discussed the relations between Painlev\'{e} equations and the recurrence coefficients of semi-classical orthogonal polynomials; see also \cite{Filipuk2,VanAssche}. In addition, the Painlev\'{e} V equation also appears in other problems with different weights; cf., e.g., \cite{BCE,ChenDai,Claeys,Min2020,MLC}.
\end{remark}
\begin{remark}
According to Theorem 7.6 in \cite{Clarkson} (see also Section 6.2.4 in \cite{VanAssche}), the Painlev\'{e} V equation (\ref{pv}) with parameters (\ref{pa}) has special function solutions. The special functions are the Kummer functions $M(a,b,t)$ and $U(a,b,t)$, which are given in terms of the confluent hypergeometric function as $M(a,b,z)={}_{1}F_{1}(a;b;z)$ and (\ref{kum}), respectively. This fact is consistent with our result, since $W_n(t)=\frac{2n+1+2\al+R_n(t)}{2n+1+2\al}$ and $R_n(t)$ can be expressed in terms of $U(a,b,t)$ from the definition (\ref{Rnt}) by using (\ref{mom}); see also (\ref{R0}).
\end{remark}

We define a quantity allied to the logarithmic derivative of the Hankel determinant,
\be\label{def}
\s_n(t):=2t\frac{d}{dt}\ln D_n(t).
\ee
It is easy to see from (\ref{hankel}) and (\ref{eq1}) that
$$
\s_n(t)=-\sum_{j=0}^{n-1}R_j(t).
$$
Then, we have the following relation satisfied by $\s_n(t),\;\bt_n$ and $r_n(t)$ from (\ref{s2p3}):
\be\label{se}
n(n+2\al-2t)-2(n+\al)r_n(t)-\s_n(t)-(2n+1+2\al)(2n-1+2\al)\beta_{n}=0.
\ee

Following \cite{MLC} and using the similar method, we obtain the next theorems and the proof will not be provided.
\begin{theorem}
The quantity $\s_n(t)$ satisfies the second-order nonlinear differential equation
\begin{small}
\bea\label{ode}
&&\Big\{t^4 (\sigma _n'')^2+2t^2\left[(n+\al)^2+ t(\al-\s_n')\right]\s_n''+2 t^3 (\sigma _n')^3-t^2\left[(t+\al)^2-3n(n+2\al)-1+4\s_n\right](\sigma _n')^2\nonumber\\
&-&2t\left[3 n^4+12   n^3\alpha+n^2(t^2+4   t\alpha+15 \alpha ^2+2)+2   n\alpha (t^2+4 t\alpha+3 \alpha ^2+2)+\alpha  (t+2 \alpha)\right]\sigma _n'\nonumber\\
&+&2t\left[3(n+\al)^2+t(t+4\al)\right]\s_n\s_n'-\big[2 n^4+8   n^3\alpha+2 n^2 (t^2+3   t\alpha+6 \alpha ^2)+4   n\alpha (t+\al)(t+2\al)\nonumber\\
&+&2 \alpha  (t+\alpha)^3\big]\s_n+2 n^6+12  n^5 \alpha+2 (t^2+3   t\alpha+14 \alpha ^2+1)n^4+8 n^3\alpha  (t^2+3   t\alpha+4 \alpha ^2+1)\nonumber\\
&+&n^2\left[2   t^3\alpha+(14 \alpha ^2+1) t^2+2t\al(15 \alpha ^2+2 )+6\al^2(3\al^2+2)\right]+n\big[4  t^3\alpha ^2+2t^2\al(6 \alpha ^2+1 )\nonumber\\
&+&4t\alpha ^2(3 \alpha ^2+2) +4\al^3(\al^2+2)\big]+2 \alpha ^2 (t+\alpha)^2\Big\}^2=4 (n+\alpha)^2 \left[(n+\al)^2+t(t+2\al)-2 t \sigma_n'\right]\nonumber\\
&\times&\Big\{t^2 \sigma _n''-t (2 n^2+4 n\alpha +1-2 \sigma _n)\s_n'-\left[(n+\al)^2+t(t+2\al)\right]\s_n+n^4+4   n^3\alpha+\alpha  (t+\alpha)\nonumber\\
&+&n^2 (t^2+2   t\alpha+5 \alpha ^2+1)+2  n\alpha  \left[(t+\al)^2+1\right]\Big\}^2,
\eea
\end{small}
and also admits the following second-order nonlinear difference equation
\begin{small}
\bea\label{snd}
&&\Big\{\big[n (n+2 \alpha-2 t)-\sigma_n\big]f(\s_n,\s_{n\pm 1})-2nt(2n-1+2\al)(2n+1+2\al)\Big\}^2\nonumber\\
&+&4t(n+\al)g(\s_n,\s_{n\pm 1})\Big\{\big[n (n+2 \alpha-2 t)-\sigma_n\big]f(\s_n,\s_{n\pm 1})-2nt(2n-1+2\al)(2n+1+2\al)\Big\}\nonumber\\
&-&4(n+\al)^2(\s_{n-1}-\s_{n})(\s_n-\s_{n+1})(n^2+2n\al-\s_n)g(\s_n,\s_{n\pm 1})=0,
\eea
\end{small}
where
\begin{small}
$$
f(\s_n,\s_{n\pm 1}):=( 2 n+1+2 \alpha)\sigma_{n-1}-(2n-1+2 \alpha ) \sigma_{n+1}-\sigma_{n-1}\sigma_{n+1}
-\s_n^2+\sigma_n (\sigma_{n-1}+\sigma_{n+1}-2),
$$
$$
g(\s_n,\s_{n\pm 1}):=(2n-1+2\al+\s_{n-1}-\s_n)(2n+1+2\al+\s_{n}-\s_{n+1}).
$$
\end{small}
\end{theorem}

\begin{theorem}
The quantity $\s_n(t)$ can be expressed in terms of the $\s$-function of a Painlev\'{e} V as follows:
$$
\s_{2n}(t)=2\tilde{H}_{n}\left(t,\al,\frac{1}{2}\right)+2\tilde{H}_{n}\left(t,\al,-\frac{1}{2}\right)+4n(n+\al),
$$
$$
\s_{2n+1}(t)=2\tilde{H}_{n}\left(t,\al,\frac{1}{2}\right)+2\tilde{H}_{n+1}\left(t,\al,-\frac{1}{2}\right)+(2n+1)(2n+1+2\al).
$$
Here $\tilde{H}_{n}(t,\al,\bt)$ satisfies the Jimbo-Miwa-Okamoto $\s$-form of Painlev\'{e} V \cite[(C.45)]{Jimbo1981},
$$
\big(t\tilde{H}_{n}''\big)^2=\big[\tilde{H}_{n}-t\tilde{H}_{n}'+2\big(\tilde{H}_{n}'\big)^2+(\nu_{0}+\nu_{1}+\nu_{2}+\nu_{3})\tilde{H}_{n}'\big]^2-4\big(\nu_{0}+\tilde{H}_{n}'\big)
\big(\nu_{1}+\tilde{H}_{n}'\big)\big(\nu_{2}+\tilde{H}_{n}'\big)\big(\nu_{3}+\tilde{H}_{n}'\big),
$$
with parameters $\nu_{0}=0,\; \nu_{1}=-(n+\al+\bt),\; \nu_{2}=n,\; \nu_{3}=-\bt$.
\end{theorem}

\section{Asymptotics of the Recurrence Coefficient, Sub-leading Coefficient and the Hankel Determinant}
Based on Dyson's Coulomb fluid approach \cite{Dyson}, for sufficiently large $n$, the eigenvalues of the $n\times n$ Hermitian matrices from a unitary ensemble with weight $w(x)$ can be approximated as a continuous fluid with a density $\sigma(x)$ supported in $J$ (a subset of $\mathbb{R}$).

Following Chen and Ismail \cite{ChenIsmail}, the equilibrium density $\sigma(x)$ is obtained by minimizing the free energy functional
\be\label{fe1}
F[\s]:=\int_{J}\s(x)\mathrm{v}(x)dx-\int_{J}\int_{J}\s(x)\ln|x-y|\s(y)dxdy
\ee
subject to
\be\label{con}
\int_{J}\s(x)dx=n.
\ee
Here $\mathrm{v}(x)=-\ln w(x)$ is the potential.

Upon minimization, the density $\s(x)$ is found to satisfy the integral equation
\be\label{ie}
\mathrm{v}(x)-2\int_{J}\ln|x-y|\s(y)dy=A,\qquad x\in J,
\ee
where $A$ is the Lagrange multiplier for the constraint (\ref{con}). Note that $A$ is a constant independent of $x$ but it depends on $n$ and the parameters in $\mathrm{v}(x)$.

By taking a derivative with respect to $x$, equation (\ref{ie}) is converted into the following singular integral equation,
\be\label{sie}
\mathrm{v}'(x)-2P\int_{J}\frac{\sigma(y)}{x-y}dy=0,\qquad x\in J,
\ee
where $P$ denotes the Cauchy principal value.

When the potential $\mathrm{v}(x)$ is convex and $\mathrm{v}''(x)>0$ in a set of positive measure, $\sigma(x)$ is supported in a single interval $(a,b)$ \cite{ChenIsmail}. In this case, the solution of (\ref{sie}) subject to the boundary condition $\sigma(a)=\sigma(b)=0$ reads,
\be\label{sigma}
\sigma(x)=\frac{\sqrt{(b-x)(x-a)}}{2\pi^2}P\int_{a}^{b}\frac{\mathrm{v}'(x)-\mathrm{v}'(y)}{(x-y)\sqrt{(b-y)(y-a)}}dy
\ee
with two supplementary conditions
\be\label{sup1}
\int_{a}^{b}\frac{\mathrm{v}'(x)}{\sqrt{(b-x)(x-a)}}dx=0,
\ee
\be\label{sup2}
\int_{a}^{b}\frac{x\:\mathrm{v}'(x)}{\sqrt{(b-x)(x-a)}}dx=2\pi n.
\ee
The endpoints of the support of the density, $a$ and $b$, are determined by (\ref{sup1}) and (\ref{sup2}). So, $a, b$, and then $\s(x)$ and $F[\s]$ all depend on $n$. Moreover, $F[\s]$ and $A$ have the following relation \cite[(2.14)]{ChenIsmail}
\be\label{fa}
\frac{\partial F}{\partial n}=A.
\ee

The rest of this section is devoted to derive the large $n$ asymptotics of the recurrence coefficient, the sub-leading coefficient for the monic orthogonal polynomials and the associated Hankel determinant with our weight (\ref{wei}). Note that, the asymptotic expansions in the following discussions are only valid for $t>0$.

For our problem, it is easy to see that the potential $\mathrm{v}(x)$ in (\ref{pt}) is even and
$$
\mathrm{v}''(x)=\frac{2\left[\al(1-x^4)+t(1+3x^2)\right]}{(1-x^2)^3}>0,\qquad x\in (-1,1).
$$
This leads to the fact that the support should be a symmetric single interval, i.e., $a=-b,\; 0<b<1$.
From (\ref{sigma}) and using (\ref{vp}), we find after some elementary computations,
\be\label{sexp}
\s(x)=\frac{\sqrt{b^2-x^2} \left[2t-b^2 t(1+x^2)+2 \alpha (1-b^2)(1-x^2)\right]}{2 \pi  (1-b^2)^{3/2} (1-x^2)^2}.
\ee
Next, we evaluate the Lagrange multiplier $A$ in the following lemma. Note that we will not use the expression of $\s(x)$ in (\ref{sexp}) to compute $A$.
\begin{lemma}
We have
\be\label{ae}
A=\frac{t}{\sqrt{1-b^2}}-2n \ln\frac{b}{2}+2\al\ln\frac{2}{1+\sqrt{1-b^2}}.
\ee
\end{lemma}
\begin{proof}
We start from writing (\ref{ie}) as
$$
\mathrm{v}(x)-2\int_{-b}^b\ln|x-y|\s(y)dy=A.
$$
Multiplying both sides by $\frac{1}{\sqrt{b^2-x^2}}$ and integrating from $-b$ to $b$ give rise to
\be\label{af}
\int_{-b}^{b}\frac{\mathrm{v}(x)}{\sqrt{b^2-x^2}}dx-2\int_{-b}^{b}dy\:\s(y)\int_{-b}^{b}\frac{\ln|x-y|}{\sqrt{b^2-x^2}}dx=\pi A,
\ee
where use has been made of the integral formula
$$
\int_{-b}^{b}\frac{dx}{\sqrt{b^2-x^2}}=\pi.
$$
Note that
\be\label{cons}
\int_{-b}^{b}\frac{\ln|x-y|}{\sqrt{b^2-x^2}}dx=C,\qquad y\in (-b,b),
\ee
where $C$ is a constant independent of $y$. This is because we have the formula
$$
P\int_{-b}^{b}\frac{1}{(x-y)\sqrt{b^2-x^2}}dx=0.
$$
Hence, we can replace $y$ by $b$ in (\ref{cons}) to compute $C$:
\bea\label{cons1}
C&=&\int_{-b}^{b}\frac{\ln(b-x)}{\sqrt{b^2-x^2}}dx\nonumber\\
&=&\int_{0}^{1}\frac{\ln[2b(1-s)]}{\sqrt{s(1-s)}}ds\nonumber\\
&=&\pi\ln\frac{b}{2},
\eea
where we have used the formula
$$
\int_{0}^{1}\frac{\ln(1-s)}{\sqrt{s(1-s)}}ds=-2\pi\ln 2.
$$
With $\mathrm{v}(x)=\frac{t}{1-x^{2}}-\al\ln(1-x^2)$, we find after some elementary computations,
\be\label{vf}
\int_{-b}^{b}\frac{\mathrm{v}(x)}{\sqrt{b^2-x^2}}dx=\frac{\pi \: t}{\sqrt{1-b^2}}+2 \pi \alpha\:  \ln \left(\frac{2}{1+\sqrt{1-b^2}}\right).
\ee
Substituting (\ref{cons}), (\ref{cons1}), (\ref{vf}) into (\ref{af}), and using the fact (\ref{con}), we establish the lemma.
\end{proof}
\begin{remark}
The integral identities used in the proof of the above lemma can be found in \cite{Gradshteyn}.
\end{remark}
In the following, we would like to derive the result of $b$ in our problem. Note that equation (\ref{sup1}) always holds in the even potential case, and (\ref{sup2}) becomes
$$
\int_{-b}^{b}\frac{2x^2\left[t+\al(1-x^2)\right]}{(1-x^2)^2\sqrt{b^2-x^2}}dx=2\pi n.
$$
This gives an equation satisfied by $b$:
$$
\frac{b^2 t+2\al(1-b^2)\left(1-\sqrt{1-b^2}\right)}{(1-b^2)^{\frac{3}{2}}}=2n.
$$
Let
\be\label{ud}
u=\sqrt{1-b^2}.
\ee
Then $u$ satisfies the following cubic equation
$$
2(n+\al)u^3+(t-2\al)u^2-t=0.
$$
It has only one real solution,
\be\label{u}
u=\frac{1}{6(n+\al)}\left[2\al-t+\sqrt[3]{\xi}+\frac{(2\al-t)^2}{\sqrt[3]{\xi}}\right],
\ee
where
\begin{small}
$$
\xi=8 \alpha ^3+6t \left(9 n^2+18  n\alpha+7 \alpha ^2\right)+6   t^2\alpha-t^3+6(n+\al)\sqrt{3t\left[27n(n+2\al) t-(t-8\al)(t+\al)^2\right]}.
$$
\end{small}
\quad Substituting the expression of $b$ in terms of $n,\; t$ and $\al$ from (\ref{ud}) and (\ref{u}) into (\ref{ae}) and letting $n\rightarrow\infty$, we find
\bea\label{A}
A&=&n \ln 4+\frac{3 t^{2/3} \sqrt[3]{n}}{2^{2/3}}+\alpha  \ln 4+\frac{ \sqrt[3]{t}\: (t-8 \alpha )}{4 \sqrt[3]{2n}}+\frac{t^{2/3}\alpha}{(2n)^{2/3}}-\frac{\alpha ^2}{3 n}-\frac{ \alpha  \sqrt[3]{t}\: (t-8 \alpha )}{12 \sqrt[3]{2}\:n^{4/3}}\nonumber\\
&-&\frac{5 t^3-48   t^2\alpha+960\alpha^2 t+320 \alpha ^3}{2160\times 2^{2/3} \sqrt[3]{t}\:n^{5/3}}+\frac{\alpha ^3}{3 n^2}+O(n^{-7/3}).
\eea
Then, it follows from (\ref{fa}) that the free energy $F[\s]$ has the following expansion as $n\rightarrow\infty$,
\bea\label{fe}
F[\s]&=&n^2 \ln 2+\frac{9\:  t^{2/3}n^{4/3}}{4\times2^{2/3}}+n\al\ln 4+\frac{3  \sqrt[3]{t}\: (t-8 \alpha )n^{2/3}}{8 \sqrt[3]{2}}+\frac{3  t^{2/3}\al\sqrt[3]{n}}{2^{2/3}}-\frac{\alpha ^2}{3}\ln n \nonumber\\
&+&C_0(t,\al)+\frac{\alpha  \sqrt[3]{t}\: (t-8 \alpha )}{4 \sqrt[3]{2n}}+\frac{5 t^3-48  t^2\al+960\alpha ^2 t+320 \alpha^3}{1440 \sqrt[3]{t} \:(2n)^{2/3}}+O(n^{-1}),
\eea
where $C_0(t,\al)$ is a constant independent of $n$.

According to formula (2.27) in \cite{ChenIsmail}, we have
$$
\bt_n=\frac{b^2}{4}\left(1+O\left(\frac{\partial^{4}F}{\partial n^{4}}\right)\right)=\frac{1-u^2}{4}\left(1+O\left(\frac{\partial^{3}A}{\partial n^{3}}\right)\right),\qquad n\rightarrow\infty,
$$
where use has been made of (\ref{fa}).
Taking account of (\ref{u}), we find that
\bea
\frac{1-u^2}{4}
&=&\frac{1}{4}-\frac{t^{2/3}}{4\times(2n)^{2/3}}+\frac{ \sqrt[3]{t} (t-2 \alpha )}{12 \sqrt[3]{2}\:n^{4/3}}+\frac{t^{2/3}\alpha}{6\times 2^{2/3}\:n^{5/3}}-\frac{(t-2 \alpha )^2}{48 n^2}\nonumber\\
&-&\frac{ \alpha  \sqrt[3]{t} (t-2 \alpha )}{9 \sqrt[3]{2}\:n^{7/3}}+\frac{5  (t^3-6t^2\alpha-6  t\alpha ^2-8 \alpha ^3)}{648\times 2^{2/3} \sqrt[3]{t}\:n^{8/3}}+O\left(\frac{1}{n^3}\right), \qquad n\rightarrow\infty.\nonumber
\eea
In view of (\ref{A}), it follows that $\bt_n$ has an expansion of the form
\be\label{exp}
\bt_n=a_0+\sum_{j=1}^{\infty}\frac{a_j}{n^{j/3}}, \qquad\qquad n\rightarrow\infty,
\ee
where
$$
a_0=\frac{1}{4},
$$
and $a_j,\; j=1,2,\ldots$, are the expansion coefficients to be determined. By using the second-order difference equation satisfied by $\bt_n$, we obtain the complete asymptotic expansion of $\bt_n$ in the next theorem.
\begin{theorem}\label{thm}
The recurrence coefficient $\bt_n(t)$ has the following large $n$ expansion:
\be\label{bte}
\bt_n(t)=\frac{1}{4}+\sum_{j=1}^{\infty}\frac{a_j}{n^{j/3}}, \qquad\qquad n\rightarrow\infty,
\ee
where the first few terms of expansion coefficients are
\bea
&&a_1=0,\qquad\qquad\qquad\qquad\qquad a_2=-\frac{t^{2/3}}{4\times2^{2/3}},\nonumber\\
&&a_3=0,\qquad\qquad\qquad\qquad\qquad  a_4=\frac{\sqrt[3]{t} (t-2 \alpha )}{12 \sqrt[3]{2}},\nonumber\\
&&a_5=\frac{t^{2/3}\alpha}{6\times2^{2/3}},\qquad\qquad\qquad\quad a_6=\frac{5-3 (t-2 \alpha )^2}{144},\nonumber\\
&&a_7=-\frac{\alpha  \sqrt[3]{t} (t-2 \alpha )}{9 \sqrt[3]{2}},\qquad\qquad\: a_8=\frac{5 \left[2 t^3-12   t^2\alpha-t(12\alpha ^2+17)-16\al(\al^2-1)\right]}{1296\times2^{2/3} \sqrt[3]{t}},\nonumber
\eea
with more terms easily computable.
\end{theorem}
\begin{proof}
Substituting the expansion (\ref{exp}) into the difference equation (\ref{btd}), and taking a large $n$ limit, we have an expression of the form
$$
e_2 n^2+e_{5/3}n^{5/3}+e_{4/3}n^{4/3}+\sum_{j=-\infty}^{3}e_{j/3}n^{j/3}=0,
$$
where each $e_{j/3}$ depends on the expansion coefficients $a_j, t$ and $\al$. In order to satisfy the above equation, all the coefficients of powers of $n$ are identically zero.
The equation $e_2=0$ reads,
$$
-4(4 a_0-1)^3(16 a_0^3-8 a_0^2+a_0-t^2)=0,
$$
which holds identically by the fact $a_0=\frac{1}{4}$.

Setting $a_0=\frac{1}{4}$ leads to $e_{5/3}$ and $e_{4/3}$ vanishing identically. The equation $e_1=0$ gives rise to
$$
256t^2 a_1^3=0.
$$
Since $t>0$, we have
$$
a_1=0.
$$
Setting $a_0=\frac{1}{4}$ and $a_1=0$ leads to $e_{2/3}$ and $e_{1/3}$ vanishing identically. The equation $e_0=0$ gives
$$
256 t^2 a_2^3 +t^4=0,
$$
we find
$$
a_2=-\frac{t^{2/3}}{4\times2^{2/3}}.
$$
With the values of $a_0, a_1$ and $a_2$, the equation $e_{-1/3}=0$ shows
$$
12\times2^{2/3} t^{10/3}a_3 =0,
$$
we have
$$
a_3=0.
$$
With the above $a_0, a_1, a_2$ and $a_3$, the equation $e_{-2/3}=0$ gives rise to
$$
\sqrt[3]{2}\: t^{10/3} \left[12 \sqrt[3]{2}\: a_4-\sqrt[3]{t} (t-2 \alpha )\right]=0,
$$
we obtain
$$
a_4=\frac{\sqrt[3]{t} (t-2 \alpha )}{12 \sqrt[3]{2}}.
$$

This procedure can be easily extended to find higher coefficients $a_5, a_6, a_7,\ldots$ We only list some of them:
$$
a_5=\frac{t^{2/3}\alpha}{6\times2^{2/3}},
$$
$$
a_6=\frac{5-3 (t-2 \alpha )^2}{144},
$$
$$
a_7=-\frac{\alpha  \sqrt[3]{t} (t-2 \alpha )}{9 \sqrt[3]{2}},
$$
$$
a_8=\frac{5 \left[2 t^3-12   t^2\alpha-t(12\alpha ^2+17)-16\al(\al^2-1)\right]}{1296\times2^{2/3} \sqrt[3]{t}}.
$$
This completes the proof.
\end{proof}
\begin{remark}
The difference equation method has also been used to derive the large $n$ asymptotic expansion of the recurrence coefficient $\bt_n$ in the generalized Freud weight problems; see Clarkson and Jordaan \cite{Clarkson1,Clarkson2}.
\end{remark}

Since
\be\label{be3}
\bt_n=\mathrm{p}(n,t)-\mathrm{p}(n+1,t),
\ee
and in view of the asymptotic expansion of $\bt_n$ in (\ref{bte}), we have the following large $n$ expansion form for $\mathrm{p}(n,t)$:
\be\label{pne}
\mathrm{p}(n,t)=b_{-3} n+b_{-2} n^{2/3}+b_{-1} n^{1/3}+b_0+\sum_{j=1}^{\infty}\frac{b_j}{n^{j/3}}.
\ee
Substituting (\ref{pne}) into (\ref{be3}) and taking a large $n$ limit, we find $$b_{-3}=-\frac{1}{4}.$$
Using the second-order difference equation satisfied by $\mathrm{p}(n,t)$, we obtain the following result.
\begin{theorem}
The sub-leading coefficient $\mathrm{p}(n,t)$ has the following expansion as $n\rightarrow\infty$:
\be\label{pnte}
\mathrm{p}(n,t)=b_{-3} n+b_{-2} n^{2/3}+b_{-1} n^{1/3}+b_0+\sum_{j=1}^{\infty}\frac{b_j}{n^{j/3}},
\ee
where
\bea
&&b_{-3}=-\frac{1}{4},\qquad\qquad\qquad\qquad\qquad\qquad b_{-2}=0,\nonumber\\
&&b_{-1}=\frac{3\:t^{2/3}}{4\times2^{2/3}},\qquad\qquad\qquad\qquad\qquad  b_0=\frac{2 \alpha+1 -4 t}{8},\nonumber\\
&&b_1=\frac{\sqrt[3]{t}\: (t-2 \alpha )}{4 \sqrt[3]{2} },\qquad\qquad\qquad\quad\quad\:\:\:\:\: b_2=\frac{(2 \alpha -1) t^{2/3}}{8\times2^{2/3}},\nonumber\\
&&b_3=\frac{5-3(t-2\al)^2}{144},\qquad\qquad\qquad\qquad b_4=\frac{(2 \alpha -1) \sqrt[3]{t}\: (2 \alpha-t )}{24 \sqrt[3]{2}},\nonumber
\eea
are the expansion coefficients of the first few terms.
\end{theorem}
\begin{proof}
Substituting (\ref{pne}) into the difference equation (\ref{pnd}) and letting $n\rightarrow\infty$, we have
$$
l_{8/3}n^{8/3}+l_{7/3}n^{7/3}+\sum_{j=-\infty}^{6}l_{j/3}n^{j/3}=0,
$$
where the expressions of $l_{j/3},\; j=8, 7, 6, \ldots$, depend on the expansion coefficients $b_j,\; t$ and $\al$. Then each $l_{j/3}$ should be identically zero. The equation $l_{8/3}=0$ reads,
$$
\frac{4}{3} (4 b_{-3}+1)^2\: b_{-2}=0,
$$
which holds identically by the fact that $b_{-3}=-\frac{1}{4}$. Setting $b_{-3}=-\frac{1}{4}$ leads to $l_{7/3}=0$ identically. Then $l_{2}=0$ shows
$$
\frac{400 }{27}b_{-2}^3=0.
$$
We have
$$
b_{-2}=0.
$$
With $b_{-3}=-\frac{1}{4}$ and $b_{-2}=0$, $l_{5/3}$ and $l_{4/3}$ vanish identically. The equation $l_1=0$ gives
$$
\frac{512 }{27}b_{-1}^3-2 t^2=0.
$$
We get
$$
b_{-1}=\frac{3\:t^{2/3}}{4\times2^{2/3}}.
$$
Following the same procedure in the proof of Theorem \ref{thm}, we can find higher coefficients easily. Finally we establish the theorem.
\end{proof}
\begin{remark}
For consistency, substituting (\ref{pnte}) into $\bt_n=\mathrm{p}(n,t)-\mathrm{p}(n+1,t)$ and taking a large $n$ limit, we find that $\bt_n$ has the same expansion as (\ref{bte}).
\end{remark}

Next, we consider $\s_n(t)$, which is related to the Hankel determinant in (\ref{def}). Before deriving the large $n$ asymptotics of $\s_n(t)$, we present a lemma below.
\begin{lemma}
The quantity $\s_n(t)$ can be expressed in terms of $\mathrm{p}(n,t)$ and $\mathrm{p}(n+1,t)$ as follows:
\be\label{sig}
\s_n(t)=-n(n+2t)-(2n-1+2\al)\mathrm{p}(n,t)-(2n+1+2\al)\mathrm{p}(n+1,t).
\ee
\end{lemma}
\begin{proof}
Substituting (\ref{s2}) into (\ref{se}) to eliminate $r_n(t)$, we have
$$
\s_n(t)=-n(n+2t)+(2n+1+2\al)\bt_n-4(n+\al)\mathrm{p}(n,t).
$$
Using the fact that $\bt_n=\mathrm{p}(n,t)-\mathrm{p}(n+1,t)$, we arrive at (\ref{sig}).
\end{proof}

From the above lemma, we are ready to obtain the large $n$ expansion of $\s_n(t)$.

\begin{theorem}\label{thm1}
The quantity $\s_n(t)=2t\frac{d}{dt}\ln D_n(t)$ has the following large $n$ asymptotic expansion:
\bea\label{snt}
\s_n(t)&=&-\frac{3 t^{2/3}n^{4/3}}{2^{2/3}}-\frac{ \sqrt[3]{t} (t-2 \alpha )n^{2/3}}{\sqrt[3]{2}}-2^{4/3}t^{2/3}\al \sqrt[3]{n}+\frac{3 t^2+60t\alpha-24\alpha ^2+4}{36}\nonumber\\
&-&\frac{2^{2/3} \alpha  \sqrt[3]{t} \:(t-2 \alpha )}{3\sqrt[3]{n}}-\frac{ t^3-6   t^2\alpha+48  t\alpha ^2+2 t-8 \alpha ^3+8 \alpha}{54 \sqrt[3]{t}\:(2n)^{2/3}}+O(n^{-1}).
\eea
\end{theorem}

\begin{proof}
Substituting (\ref{pnte}), the expansion of $\mathrm{p}(n,t)$, into (\ref{sig}), we obtain the desired result by taking a large $n$ limit.
\end{proof}
\begin{remark}
If we assume $\s_n(t)=\sum_{j=-\infty}^{4}d_jn^{j/3}$ as $n\rightarrow\infty$, then we can obtain the expansion coefficients $d_j$ from (\ref{snd}), the second-order difference equation satisfied by $\s_n(t)$. We find that this agrees precisely with the result in Theorem \ref{thm1}.
\end{remark}
Finally, we have the large $n$ asymptotic expansion of $D_n(t)$.

\begin{theorem}
The Hankel determinant $D_n(t)$ has the following expansion as $n\rightarrow\infty$:
\bea
\ln D_n(t)&=&-n^2 \ln 2-\frac{9\:  t^{2/3}n^{4/3}}{4\times2^{2/3}}-\tilde{c}_{1}(\al) n-\frac{3  \sqrt[3]{t}\: (t-8 \alpha )n^{2/3}}{8 \sqrt[3]{2}}-\frac{3 t^{2/3}\alpha\sqrt[3]{n}}{2^{2/3}} \nonumber\\
&+&\frac{(12 \alpha ^2-5) \ln n}{36}+\frac{3t^2+120t\al-8 (6 \alpha ^2-1) \ln t}{144}-\tilde{c}_0(\al)-\frac{\alpha  \sqrt[3]{t}\: (t-8 \alpha )}{4 \sqrt[3]{2n}}\nonumber\\
&-&\frac{5 t^3-48  t^2\al+40 (24 \alpha ^2+1) t+320 \alpha(\alpha ^2-1)}{1440 \sqrt[3]{t} \:(2n)^{2/3}}+O(n^{-1}),\nonumber
\eea
where $\tilde{c}_0(\al)$ and $\tilde{c}_{1}(\al)$ are constants depending on $\al$ only.
\end{theorem}
\begin{proof}
Let
$$
F_n(t):=-\ln D_n(t)
$$
be the ``free energy''.
From (\ref{re}) we have
\be\label{dc}
-\ln\bt_n=F_{n+1}(t)+F_{n-1}(t)-2F_n(t).
\ee
For sufficiently large $n$, Chen and Ismail \cite{ChenIsmail} showed that $F_n(t)$ is approximated by the free energy $F[\s]$ defined in (\ref{fe1}). Taking account of (\ref{fe}),
we have the following large $n$ expansion form for $F_n(t)$:
\be\label{fna}
F_n(t)=c(t,\al)\ln n+\sum_{j=-\infty}^{6}c_{j/3}(t,\al)n^{j/3}.
\ee
Substituting (\ref{bte}) and (\ref{fna}) into the difference equation (\ref{dc}) and letting $n\rightarrow\infty$, we obtain the asymptotic expansion of $F_n(t)$ by equating coefficients of powers of $n$:
\begin{small}
\bea\label{fe2}
F_n(t)&=&n^2 \ln 2+\frac{9\:  t^{2/3}n^{4/3}}{4\times2^{2/3}}+c_{1}(t,\al) n+\frac{3  \sqrt[3]{t}\: (t-8 \alpha )n^{2/3}}{8 \sqrt[3]{2}}+\frac{3 t^{2/3}\alpha\sqrt[3]{n}}{2^{2/3}}-\frac{(12 \alpha ^2-5) \ln n}{36} \nonumber\\
&+&c_0(t,\al)+\frac{\alpha  \sqrt[3]{t}\: (t-8 \alpha )}{4 \sqrt[3]{2n}}+\frac{5 t^3-48  t^2\al+40 (24 \alpha ^2+1) t+320 \alpha(\alpha ^2-1)}{1440 \sqrt[3]{t} \:(2n)^{2/3}}+O(n^{-1}),
\eea
\end{small}
where $c_{1}(t,\al)$ and $c_0(t,\al)$ are undetermined coefficients independent of $n$.
Hence,
\begin{small}
\bea\label{dn}
\ln D_n(t)&=&-n^2 \ln 2-\frac{9\:  t^{2/3}n^{4/3}}{4\times2^{2/3}}-c_{1}(t,\al) n-\frac{3  \sqrt[3]{t}\: (t-8 \alpha )n^{2/3}}{8 \sqrt[3]{2}}-\frac{3   t^{2/3}\alpha\sqrt[3]{n}}{2^{2/3}}+\frac{(12 \alpha ^2-5) \ln n}{36} \nonumber\\
&-&c_0(t,\al)-\frac{\alpha  \sqrt[3]{t}\: (t-8 \alpha )}{4 \sqrt[3]{2n}}-\frac{5 t^3-48  t^2\al+40 (24 \alpha ^2+1) t+320 \alpha(\alpha ^2-1)}{1440 \sqrt[3]{t} \:(2n)^{2/3}}+O(n^{-1}).
\eea
\end{small}\\[-40pt]

In order to know more information of $c_{1}(t,\al)$ and $c_{0}(t,\al)$, we take a derivative with respect to $t$ in (\ref{dn}) and substitute it into (\ref{def}), to find
\bea
\s_n(t)&=&-\frac{3 t^{2/3}n^{4/3}}{2^{2/3}}-2nt\frac{d}{dt}c_{1}(t,\al)-\frac{ \sqrt[3]{t} (t-2 \alpha )n^{2/3}}{\sqrt[3]{2}}-2^{4/3}t^{2/3}\al \sqrt[3]{n}-2t\frac{d}{dt}c_0(t,\al)\nonumber\\
&-&\frac{2^{2/3} \alpha  \sqrt[3]{t} \:(t-2 \alpha )}{3\sqrt[3]{n}}-\frac{ t^3-6   t^2\alpha+48  t\alpha ^2+2 t-8 \alpha ^3+8 \alpha}{54 \sqrt[3]{t}\:(2n)^{2/3}}+O(n^{-1}).
\eea
Comparing the above with (\ref{snt}), we have
$$
\frac{d}{dt}c_{1}(t,\al)=0,
$$
$$
-2t\frac{d}{dt}c_0(t,\al)=\frac{3 t^2+60t\alpha-24\alpha ^2+4}{36}.
$$
It follows that
\be\label{c1}
c_{1}(t,\al)=\tilde{c}_{1}(\al),
\ee
\be\label{c0}
c_0(t,\al)=\frac{8 (6 \alpha ^2-1) \ln t-3 t^2-120t\al}{144}+\tilde{c}_0(\al),
\ee
where $\tilde{c}_0(\al)$ and $\tilde{c}_{1}(\al)$ are constants depending on $\al$ only.
Substituting (\ref{c1}) and (\ref{c0}) into (\ref{dn}), we establish the theorem.
\end{proof}

\begin{remark}
We can not evaluate explicitly the constants $\tilde{c}_0(\al)$ and $\tilde{c}_{1}(\al)$ with our method. However, by comparing (\ref{fe2}) with (\ref{fe}), we conjecture that $\tilde{c}_{1}(\al)=c_1(t,\al)=\al\ln 4$. For the determination of the constants in the case of regularly perturbed weights, see \cite{Basor2015} for reference.
\end{remark}

\section*{Acknowledgments}
The authors are very grateful to the editors and the anonymous reviewers for their helpful and constructive comments. The work of Chao Min was supported by the National Natural Science Foundation of China under grant number 12001212 and by the Scientific Research Funds of Huaqiao University under grant number 17BS402.
Yang Chen was supported by the Macau Science and Technology Development Fund under grant number FDCT 023/2017/A1 and by the University of Macau under grant number MYRG 2018-00125-FST.

\section*{Data Availability Statement}
Data sharing not applicable to this article as no datasets were generated or analysed during the current study.

\vspace{58pt}
\noindent School of Mathematical Sciences, Huaqiao University, Quanzhou 362021, China\\
e-mail: chaomin@hqu.edu.cn\\

\noindent Department of Mathematics, Faculty of Science and Technology, University of Macau, Macau, China\\
e-mail: yangbrookchen@yahoo.co.uk

\end{document}